\documentclass[12pt, a4paper,reqno]{amsart}  

\allowdisplaybreaks[4] 

\usepackage{url}
\usepackage{amssymb}

\newcommand{\rmi}{\mathrm{i}}

\newcommand{\ZZ}{\mathbb{Z}}

\newtheorem{theorem}{Theorem}
\newtheorem{lemma}[theorem]{Lemma}
\newtheorem{corollary}[theorem]{Corollary} 
\newtheorem{proposition}[theorem]{Proposition} 
\theoremstyle{definition}

\newtheorem{example}[theorem]{Example}
\newtheorem{remark}[theorem]{Remark} 

\newcommand{\osum}{\mathop{\oplus}}

\renewcommand{\Tilde}{\widetilde}

\newcommand{\ess}{\mathrm{ess}}

\newcommand{\cH}{\mathcal{H}}
\newcommand{\cL}{\mathcal{L}}
\newcommand{\cB}{\mathcal{B}}
\newcommand{\fh}{\mathfrak{h}}
\newcommand{\CC}{\mathbb{C}}
\newcommand{\RR}{\mathbb{R}}
\newcommand{\fn}{\mathfrak{n}}
\newcommand{\cD}{\mathcal{D}}

\DeclareMathOperator{\dom}{dom}
\DeclareMathOperator{\ran}{ran}
\DeclareMathOperator{\tr}{tr}

\begin{document}

\title{Self-adjoint indefinite Laplacians}

\author{Claudio Cacciapuoti}

\address{DiSAT, Sezione di Matematica, Universit\`a dell'Insubria,
via Valleggio 11, 22100 Como, Italy}
\email{claudio.cacciapuoti@uninsubria.it}

\author{Konstantin Pankrashkin}

\address{Laboratoire de Math\'ematiques d'Orsay, Univ.~Paris-Sud, CNRS, Universit\'e Paris-Saclay, 91405 Orsay, France}

\email{konstantin.pankrashkin@math.u-psud.fr}
\urladdr{http://www.math.u-psud.fr/~pankrash/}

\author{Andrea Posilicano}
\address{DiSAT, Sezione di Matematica, Universit\`a dell'Insubria,
via Valleggio 11, 22100 Como, Italy}
\email{andrea.posilicano@uninsubria.it}

\begin{abstract}
Let $\Omega_-$ and $\Omega_+$ be two bounded smooth  domains in $\mathbb{R}^n$, $n\ge 2$,
separated by a hypersurface $\Sigma$. For $\mu>0$, consider the function $h_\mu=1_{\Omega_-}-\mu 1_{\Omega_+}$.
We discuss self-adjoint realizations of the operator $L_{\mu}=-\nabla\cdot h_\mu \nabla$ in $L^2(\Omega_-\cup\Omega_+)$
with the Dirichlet  condition at the exterior boundary. We show that $L_\mu$
is always essentially self-adjoint on the natural domain (corresponding to transmission-type boundary conditions
at the interface $\Sigma$) and study some properties of its unique self-adjoint extension
$\mathcal{L}_\mu:=\overline{L_\mu}$. If $\mu\ne 1$, then $\mathcal{L}_\mu$ simply coincides with $L_\mu$ and has
compact resolvent. If $n=2$, then $\mathcal{L}_1$ has a non-empty essential spectrum, $\sigma_\mathrm{ess}(\mathcal{L}_{1})=\{0\}$.
If $n\ge 3$, the spectral properties of $\mathcal{L}_1$ depend on the geometry of $\Sigma$. In particular, it
has compact resolvent if $\Sigma$ is the union of disjoint strictly convex hypersurfaces,
but can have a non-empty essential spectrum if a part of $\Sigma$ is flat. 
Our construction features the method of boundary triplets, and the problem is reduced to finding
the self-adjoint extensions of a pseudodifferential operator on $\Sigma$.
We discuss some links between the resulting self-adjoint operator $\mathcal{L}_\mu$ and
some effects observed in negative-index materials.
\end{abstract}



\maketitle

\section{Introduction}

Let $n\ge 2$ and $\Omega\subset\RR^{n}$ be an open bounded set
with a smooth boundary $\partial\Omega$. Let $\Omega_{-}$ be a subset of $\Omega$ having a smooth boundary $\Sigma$
(called \emph{interface}) and such that $\overline\Omega_{-}\subset \Omega$. In addition, we consider the open set
$\Omega_{+}:=\Omega\setminus\overline{\Omega_-}$, 
whose boundary is $\partial\Omega_+=\Sigma\cup{\partial\Omega}$, and
denote by $N_\pm$ the unit normal on $\Sigma$  exterior
with respect to $\Omega_\pm$. For $\mu>0$, consider the function
$h:\Omega\setminus\Sigma\to \RR$,
\[
h_\mu(x)=\begin{cases}
1, & x\in \Omega_-,\\
-\mu, & x\in \Omega_+.
\end{cases}
\]
The aim of the present work is to construct self-adjoint
operators in $L^2(\Omega)$ corresponding to the formally symmetric differential
expression $L_{\mu}=-\nabla\,\cdot h_\mu\nabla$.
The operators of such a type appear e.g. in the study of negative-index metamaterials, and we refer to the recent paper \cite{ng}
for a survey and an extensive bibliography; we remark that the parameter $\mu$ is usually called \emph{contrast}.
A possible approach  is to consider the sesquilinear form
\[
\ell_{\mu}(u,v)=\int_{\Omega} h_\mu\overline{\nabla u}\cdot \nabla v\, dx, \quad u,v\in H^1_0(\Omega),
\]
and then to define $L_{\mu}$ as the operator generated by $\ell_{\mu}$, 
in particular, for all functions $v$ from the domain of $L_{\mu}$ one should then have
\begin{equation}
 \label{lll}
\int_\Omega \overline{u}\,L_{\mu}v\, dx=\ell_{\mu}(u,v) \,,\quad u\in H^1_0(\Omega).
\end{equation}
But the form $\ell_{\mu}$ is not semibounded below, hence, the operator obtained
in this way can have exotic properties, in particular, its self-adjointness
is not guaranteed. We refer to \cite{kostr,ss1,ss2} for some available results
in this direction.

In \cite{bk} a self-adjoint operator for the above expression
was constructed for a very particular geometry when $\Omega_{-}=(-1,0)\times(0,1)$ and $\Omega_{+}=(0,1)\times(0,1)$,
which enjoys a separation of variables
and some symmetries. An interesting feature of the model is the possibility of a non-empty
essential spectrum although the domain is bounded.
Constructing self-adjoint operators realizations of $L_\mu$ for the general case
is an open problem, see \cite{open}. 
In the present note, we give a solution in the case of a smooth interface.

One should remark that the study of various boundary value problems
involving differential expressions $\nabla\cdot h \nabla$
with sign-changing $h$ has a long history, and the most
classical form involves unbounded domains with a suitable radiation
condition at infinity, cf. \cite{cost,grieser,ola}. In particular,
the geometric conditions appearing in the main results below are
very close to those of~\cite{ng-jmpa,ola} for the well-posedness of a transmission problem.
The case of a non-smooth interface $\Sigma$, which was partially studied in \cite{bbdr,bb1,bb2},
is not covered by our approach.

In fact, the problem of self-adjoint realizations
the non-critical case $\mu\ne 1$ was essentially settled in \cite{bbdr},
while for the critical case $\mu=1$ was only studied for
the above-mentioned example of~\cite{bk},
in \cite[Chapter 8]{ss2} for another similar situation
(symmetric $\Omega_-$ and $ \Omega_+$ separated by a finite portion
of a hyperplane), and in \cite{hus} for the one-dimensional case.
In a sense, in the present work we recast some techniques
of the transmission problems and the pseudodifferential operators
into the setting of self-adjoint extensions.
Using the machinery of boundary triplets we reduce the
problem first to finding self-adjoint extensions of a symmetric differential
operator and then to the analysis of the associated Weyl function acting
on the interface $\Sigma$.
Then one arrives at the study of the essential self-adjointness of a pseudodifferential
operator on $\Sigma$, whose properties depend on the dimension.
We hope that, in view of the recent progress in the theory of self-adjoint
extensions of partial differential operators, see e.g. \cite{b2,b1,gm},
such a direct reformulation could be  a starting point for a further advance
in the study of indefinite operators.

Similar to \cite{bk}, our approach is based on the theory of self-adjoint extensions.
Using the natural identification $L^2(\Omega)\simeq L^2(\Omega_-)\oplus
L^2(\Omega_+)$, $u\simeq (u_-,u_+)$, we introduce the  sets
\begin{multline*}       
\cD^s_{\mu}(\Omega\backslash\Sigma):=\Big\{
u=(u_-,u_+)\in H^s(\Omega_-)\oplus H^s(\Omega_+):
\Delta u_{\pm}\in L^{2}(\Omega_{\pm})\,,
\\
\quad 
u_-=u_+ \ \text{ and }\ N_-\!\cdot \!\nabla u_-=\mu N_+\!\cdot\! \nabla u_+\ \text{ on } \Sigma\,,\quad
u_+=0 \ \text{ on }\ \partial\Omega \Big\}\,,\quad s\ge 0\,.
\end{multline*}
Here and below, the values at the boundary are understood as suitable Sobolev traces; the exact definitions are given in Section~\ref{sec3}.
Let us recall that for $\frac12 < s<\frac32$ and  $u=(u_{-},u_{+})\in H^s(\Omega_-)\oplus H^s(\Omega_+)$
the conditions $u_-=u_+$ on $\Sigma$ and $u_+=0$ on $\partial\Omega$ entail $u\in H^{s}_{0}(\Omega)$,
see e.g. \cite[Theorem 3.5.1]{Agra}. In particular,
\begin{equation}
 \label{cdd1}
\cD^2_{\mu}(\Omega\backslash\Sigma)\subseteq H^{\frac32-\varepsilon}_{0}(\Omega) \text{ for  $\varepsilon>0$},
\qquad
\cD^1_\mu(\Omega\backslash\Sigma)\subseteq H^{1}_{0}(\Omega).
\end{equation}

Consider the operator
\begin{equation}
   \label{opl}
L_{\mu} (u_-,u_+)=(-\Delta u_-, \mu \Delta u_+),
\text{ with }
\Delta=\dfrac{\partial^2}{\partial x_1^2}+\dots+\dfrac{\partial^2}{\partial x_n^2},
\end{equation}
acting on the domain 
\begin{equation}\label{doml}
\dom L_{\mu}=\cD^{2}_{\mu}(\Omega\backslash\Sigma)\,.
\end{equation} 
Remark that $L_{\mu}$ satisfies \eqref{lll} and it is a densely defined symmetric operator
in $L^2(\Omega)$. Therefore, we use $L_{\mu}$ as a starting point and seek its self-adjoint extensions.
Even if the case $\mu\ne 1$ was studied earlier, we include it into consideration
as it does not imply any additional difficulties.

\begin{theorem}[Self-adjointness]\label{thm1}
The operator $L_\mu$ is essentially self-adjoint, and we denote
\[
\cL_\mu:=\overline{L_\mu}
\]
its closure and unique self-adjoint extension.
Furthermore, if $\mu\ne 1$, then $\cL_\mu =L_\mu$, i.e. $L_\mu$ itself is self-adjoint, 
and has compact resolvent.
\end{theorem}

Now we consider in greater detail the critical case $\mu=1$. The properties
of $\cL_1$ appear to depend on the dimension.
In two dimensions, we have a complete result:

\begin{theorem}[Critical contrast in two dimensions]\label{thm2}
Let $\mu=1$ and $n=2$, then
\begin{equation}
      \label{dom2}
\dom \cL_1=\cD^{0}_{1}(\Omega\backslash\Sigma),
\quad
\cL_1(u_-,u_+)=(-\Delta u_-,\Delta u_+),
\end{equation}
and the essential spectrum is non-empty, $\sigma_\ess(\cL_1)=\{0\}$.
\end{theorem}
Remark (see Proposition~\ref{thm4} below) that $0$ is not necessarily an eigenvalue of $\cL_1$,
contrary to the preceding examples given in \cite{bk} and \cite[Chapter 8]{ss2}
for which the essential spectrum consisted of an infinitely degenerate zero eigenvalue.

In dimensions $n\ge 3$ the result appears to depend on the geometric properties of $\Sigma$: 

\begin{theorem}[Critical contrast in dimensions $\ge 3$]\label{thm3}
Let $\mu=1$ and $n\ge 3$, then $\cL_1$ acts as
$\cL_1(u_-,u_+)=(-\Delta u_-,\Delta u_+)$, and its domain satisfies
\begin{equation}
         \label{dom3}
\cD^{1}_{1}(\Omega\backslash\Sigma)\subseteq\dom \cL_1.
\end{equation}
Furthermore,
\begin{itemize}
\item[(a)] If on each connected component of $\Sigma$
the principal curvatures are either all strictly positive or all
strictly negative (in particular, if each maximal connected component
of $\Sigma$ is strictly convex), then 
\begin{equation}
         \label{dom4}
\dom\cL_1=\cD^{1}_{1}(\Omega\backslash\Sigma)
\end{equation}
and $\cL_1$ has compact resolvent.
\item[(b)] 
If a subset of $\Sigma$ is isometric to a non-empty open subset of $\RR^{n-1}$,
then 
\begin{equation}
         \label{dom5}
\dom\cL_1\not=\cD^{s}_{1}(\Omega\backslash\Sigma)\quad \text{for any $s>0$,}
\end{equation}
the essential spectrum of $\cL_1$
is non-empty, and $\{0\}\subseteq \sigma_\ess(\cL_1)$.
\end{itemize}
\end{theorem}

The proofs of the three theorems are given in Sections~\ref{sec2}--\ref{sec4}.
In section~\ref{sec2} we recall the tools of the machinery of boundary triplets
for self-adjoint extensions of symmetric operators. In section~\ref{sec3}
we apply these tools to the operators $L_\mu$ and reduce the initial
problem to finding self-adjoint extensions of a pseudodifferential
operator $\Theta_\mu$ acting on $\Sigma$, which is essentially a linear combination
of (suitably defined) Dirichlet-to-Neumann maps on $\Omega_\pm$.
The self-adjoint extensions of $\Theta_\mu$ are studied in Section~\ref{sec4}
using a combination of some facts about Dirichlet-to-Neumann maps 
and pseudodifferential operators.

In addition, we use the definition of the operators $\cL_\mu$ to revisit
some results concerning the so-called cloaking by negative
materials, see e.g \cite{ng,cloaking}. For $\delta>0$, consider
the operator $T_{\mu,\delta}$ generated by the regularized sesqulinear form
\[
t_{\mu,\delta}(u,v):=\int_{\Omega\setminus\Sigma} \overline{\nabla u}\cdot (h_\mu+i\delta) \nabla v\, dx,
\quad u,v\in H^1_0(\Omega).
\]
By the Lax-Milgram theorem, the operator $T_{\mu,\delta}:L^2(\Omega)\supset H^1_0(\Omega)\supset\dom T_{\mu,\delta}\to L^2(\Omega)$ has
a bounded inverse, hence, for $g\in L^2(\Omega)$ one may define $u_\delta:=(T_{\mu,\delta})^{-1}g\in H^1_0(\Omega)$.
It was observed in \cite{cloaking} that the limit properties of $u_\delta$ as $\delta$ tends to $0$
can be quite irregular, in particular, the norm $\|u_\delta\|_{H^1(V)}$ may remain bounded
for some subset $V\subset\Omega$ while $\|u_\delta\|_{H^1(\Omega\setminus V)}$ goes to infinity. 
The most prominent example is as follows: for $0<r<R$ we denote
\begin{gather*}
B_r:=\big\{
x\in \RR^n: |x|<r
\big\},
\quad
B_{r,R}:=\big\{
x \in \RR^n: r<|x|<R
\big\},\quad
S_r:=\big\{
x\in  \RR^n: |x|=r
\big\},
\end{gather*}
pick three constants $0<r_i<r_e<R$ and consider the above operator $T_{\mu,\delta}$
corresponding to 
\begin{equation}
   \label{omb}
\Omega:=B_R, \quad \Omega_-:=B_{r_i,r_e},
\end{equation}
and set $u_\delta:=(T_{\mu,\delta})^{-1}g$ with $g$ supported in $B_{r_e,R}$.
Then the norm $\|u_\delta\|_{H^1(\Omega)}$ remains bounded for $\delta$ approaching $0$
provided $\mu\ne 1$ or $n\ge 3$. For $\mu=1$ and $n=2$ the situation appears to be different:
if $g$ is supported outside the ball $B_a$ with $a=r_e^2/r_i$, then $\|u_\delta\|_{H^1(B_R)}$
remains bounded, otherwise, for a generic $g$, the norm $\|u_\delta\|_{H^1(B_{r_e,R})}$
is bounded, while $\|u_\delta\|_{H^1(B_{r_i,r_e})}$ becomes infinite, see \cite{cloaking}.
Such a non-uniform blow-up the $H^1$ norm is often referred
to as an anomalously localized resonance, and we refer to \cite{amm1,amm2,bouch,kohn,ng2}
for a discussion of other similar models and generalizations.

It is interesting to understand whether similar observations can be made based
on the direct study of the operator $\cL_\mu$.
In fact, instead of taking the limit of $u_\delta$ one may study directly the solutions
 $u$ of $\cL_\mu u=g$.
If $\mu\ne 1$, then $u\in H^1_0(\Omega)$ by Theorem~\ref{thm1}.
Furthemore, due to Theorem~\ref{thm3}(b) the same holds
for $\mu=1$ and $n\ge 3$ as  the interface $\Sigma$ consists of two strictly convex
hypersurfaces (the spheres $S_{r_i}$ and $S_{r_e}$), which is quite close to the discussion
of~\cite{kett}; we remark that a separation of variables
shows that $\cL_1$ is injective and thus surjective in this case.
The study of the case $\mu=1$ and $n=2$ is more involved, and the link to
the anomalously localized resonance appears as follows (we assume
a special form of the function $g$ to make the discussion more transparent):
\begin{proposition}\label{thm4}
Let $\mu=1$ and $n=2$, then the operator $\cL_1$  associated with \eqref{omb} is injective,
and a function $g\in L^2(B_R)$ of the form
\begin{equation}
     \label{eq-gab}
g(r\cos\theta,r\sin\theta)=1_{(a,b)}(r) h(\theta), \quad h\in L^2(0,2\pi),
\quad
r_e\le a<b\le R,
\end{equation}
belongs to $\ran \cL_1$ if and only if
\begin{equation}
      \label{eq-hm}
\sum_{m\in \ZZ\setminus\{0\}} \dfrac{|h_m|^2}{|m|^5} \Big(\dfrac{r_e^2}{ r_i a}\Big)^{2|m|}<\infty
\text{ with }
h_m:=\dfrac{1}{2\pi}\int_0^{2\pi} h(\theta)e^{-\rmi m \theta}d\theta.
\end{equation}
In particular, the condition \eqref{eq-hm} is satisfied for any $h$ if $a\ge r_e^2/r_i$,
but fails generically for $a<r_e^2/r_i$.
\end{proposition}

\subsection*{Acknowledgments}
C.C. acknowledges the support of the FIR 2013 project
``Condensed Matter in Mathematical Physics'',
Ministry of University and Research of Italian Republic (code RBFR13WAET).

\section{Preliminaries}\label{sec2} 

\subsection{Boundary triplets}\label{secbt}

For a linear operator $B$ we denote $\dom B$, $\ker B$, $\ran B$, $\sigma(B)$
and $\rho(B)$ its domain, kernel, range, spectrum and resolvent set respectively. For a self-adjoint operator $B$, by $\sigma_\ess(B)$
and $\sigma_\mathrm{p}(B)$ we denote  respectively its essential spectrum and point spectrum (i.e. the set of the eigenvalues).
The scalar product in a Hilbert space $\cH$ will be denoted as $\langle\cdot,\cdot\rangle_{\cH}$
or, if there is no ambiguity, simply as $\langle\cdot,\cdot\rangle$,
and it is assumed anti-linear with respect to the first argument.
By $\cB(\fh,\cH)$ we mean the Banach space of the bounded linear operators
from a Hilbert space $\fh$ to a Hilbert space $\cH$, and we set
$\cB(\cH):=\cB(\cH,\cH)$.

Let us recall the key points of the method of boundary triplets for self-adjoint extensions~\cite{BGP08,DM91,GG}.
Our presentation mostly follows the first chapters of \cite{BGP08}.
Let $S$ be a closed densely defined symmetric operator in a Hilbert space $\cH$. A triplet
$(\fh,\Gamma_1,\Gamma_2)$, where $\fh$ is an auxiliary  Hilbert space
and $\Gamma_1$ and $\Gamma_2$ are linear maps from the domain $\dom S^*$
of the adjoint operator $S^*$ to $\fh$, is called a \emph{boundary triplet}
for $S$ if the following three conditions are satisfied:
\begin{itemize}
\item[(a)] the Green's identity
$\langle u,S^*v\rangle_\cH-\langle S^*u,v\rangle_\cH
=\langle \Gamma_1 u, \Gamma_2v\rangle_\fh-\langle \Gamma_2 u,\Gamma_1 v\rangle_\fh$ holds 
for all $u,v\in \dom S^*$,
\item[(b)] the map $\dom S^*\ni u\mapsto (\Gamma_1 u,\Gamma_2u)\in \fh\times\fh$ is surjective,
\item[(c)] $\ker \Gamma_1\mathop{\cap}\ker\Gamma_2=\dom S$.
\end{itemize}
It is known that a boundary triplet for $S$ exists if and only if $S$
admits self-adjoint extensions, i.e. if its deficiency indices
are equal, $\dim\ker(S^*- i)=\dim\ker(S^*+ i)=:\fn(S)$.
A boundary triplet is not unique, but for any choice of a boundary triplet
$(\fh,\Gamma_1,\Gamma_2)$ for $S$ one has $\dim\fh=\fn(S)$.

Let as assume from now on that the deficiency indices of $S$ are equal
and pick a boundary triplet $(\fh,\Gamma_1,\Gamma_2)$, then
the self-adjoint extensions of $S$
are in a one-to-one correspondence with the self-adjoint linear relations
in $\fh$ (multi-valued self-adjoint operators). In the present text we prefer to keep
the operator language and reformulate this result as follows, cf. \cite{Pos08}:
Let $\Pi:\fh\to\ran\Pi\subseteq\fh$ be an orthogonal projector in $\fh$
and $\Theta$ be a linear operator in the Hilbert space
$\ran\Pi$ with the induced scalar product. Denote by
$A_{\Pi,\Theta}$ the restriction of $S^*$ to
\[
\dom A_{\Pi,\Theta}=\big\{
u\in \dom S^*: \Gamma_1u \in \dom \Theta \text{ and }
\Pi \Gamma_2 u=\Theta \Gamma_1 u
\big\},
\]
then $A_{\Pi,\Theta}$ is symmetric/closed/self-adjoint
iff $\Theta$ possesses the respective property as an operator
in $\ran\Pi$, and one has $\overline{A_{\Pi,\Theta}}=
A_{\Pi,\overline{\Theta}}$, where as usual the bar means
taking the closure. Futhermore, any
self-adjoint extension of $S$ is of the form $A_{\Pi,\Theta}$.

The spectral analysis of the self-adjoint extensions
can be carried out using the associated Weyl functions.
Namely, let $A$ be the restriction of $S^*$
to $\ker\Gamma_1$, i.e. corresponds to $(\Pi,\Theta)=(0,0)$
in the above notation, which is a self-adjoint operator.
For $z\in\rho(A)$ the restriction
$\Gamma_1:\ker(S^*-z)\to \fh$ is a bijection,
and we denote its inverse by $G_z$.
The map $z\mapsto G_z$, called the associated $\gamma$-field,
is then a holomorphic map from $\rho(A)$ to $\cB(\fh,\cH)$ with
\begin{equation}
        \label{gam1}
G_z-G_w=(z-w)(A-z)^{-1}G_w,
\quad z,w\in\rho(A).
\end{equation}
The \emph{Weyl function} associated with the boundary triplet
is then the holomorphic map
\[
\rho(A)\ni z\mapsto M_z:=\Gamma_2 G_z\in \cB(\fh).
\]
To describe the spectral properties of the self-adjoint
operators $A_{\Pi,\Theta}$ let us consider first
the case $\Pi=1$, then $\Theta$ is a self-adjoint operator
in $\cH$, and the following holds, see Theorems 1.29 and
Theorem 3.3 in \cite{BGP08}:
\begin{proposition}\label{prop1}
For any $z\in \rho(A)\mathop{\cap}\rho(A_{1,\Theta})$
one has $0\in \rho(\Theta-M_z)$ and
the resolvent formula
\begin{equation}
        \label{krein1}
(A_{1,\Theta}-z)^{-1}=(A-z)^{-1}
+G_z (\Theta-M_z)^{-1}G_{\overline z}^*
\end{equation}
holds. In addition, for any $z\in\rho(A)$ one has the equivalences:
\begin{itemize}
 \item[(a)] $z\in\sigma(A_{1,\Theta})$ iff $0\in \sigma(\Theta-M_z)$,
 \item[(b)] $z\in\sigma_\ess(A_{1,\Theta})$ iff $0\in \sigma_\ess(\Theta-M_z)$,
 \item[(c)] $z\in\sigma_\mathrm{p}(A_{1,\Theta})$ iff $0\in \sigma_\mathrm{p}(\Theta-M_z)$
with $G_z$ being an isomorphism of the  eigensubspaces.
 \item[(d)] If $f\in \cH$, then $f\in \ran (A_{1,\Theta}-z)$ iff $G_{\overline z}^*f\in \ran (\Theta-M_z)$.
If $\Theta-M_z$ is injective, the resolvent formula \eqref{krein1} still holds for
such $f$.
\end{itemize}
\end{proposition}
It seems that the point (d) was not stated explicitly in earlier references, its proof is given
in Appendix~\ref{app1}.

Now let $A_{\Pi,\Theta}$ be an arbitrary self-adjoint extension. Denote by $S_\Pi$
the restriction of $S^*$ to 
\[
\dom S_\Pi=\{
u\in\dom S^*: \Gamma_1 u=\Pi \Gamma_2u=0
\},
\]
which is a closed densely defined symmetric operator whose
adjoint $S^*_\Pi$ is the restriction of $S^*$ to 
\[
\dom S^{*}_\Pi=\{
u\in\dom S^*: \Gamma_1 u\in \ran \Pi
\},
\]
 then $(\ran \Pi,\Gamma^\Pi_1,\Gamma^\Pi_2)$,
with $\Gamma^\Pi_j:=\Pi \Gamma_j$, is a boundary triplet for $S_\Pi$,
and the restriction of $S^*_\Pi$ to $\ker \Gamma^\Pi_1$
is the same operator $A$ as prevously. The associated $\gamma$-field and Weyl function
take the form 
\[
z\mapsto G^\Pi_z:=G_z\Pi^* , \quad
z\mapsto M^\Pi_z:=\Pi M_z\Pi^* ,
\]
and $\dom A_{\Pi,\Theta}:=\{u\in \dom S^*_\Pi: \Gamma^\Pi_2 u=\Theta \Gamma^\Pi_1 u\}$,
see \cite[Remark 1.30]{BGP08}. A direct application of Proposition~\ref{prop1} gives
\begin{corollary}\label{corol5}
For any $z\in \rho(A)\mathop{\cap}\rho(A_{\Pi,\Theta})$
one has $0\in \rho(\Theta-M^\Pi_z)$ and
the resolvent formula
\begin{equation}
        \label{krein2}
(A_{\Pi,\Theta}-z)^{-1}=(A-z)^{-1}
+G^\Pi_z (\Theta-M^\Pi_z )^{-1} (G^\Pi_{\overline z})^*
\end{equation}
holds, and, in addition, for any $z\in\rho(A)$ one has
\begin{itemize}
 \item[(a)] $z\in\sigma(A_{\Pi,\Theta})$ iff $0\in \sigma(\Theta-M^\Pi_z)$,
 \item[(b)] $z\in\sigma_\ess(A_{\Pi,\Theta})$ iff $0\in \sigma_\ess(\Theta-M^\Pi_z)$,
 \item[(c)] $z\in\sigma_\mathrm{p}(A_{\Pi,\Theta})$ iff $0\in \sigma_\mathrm{p}(\Theta-M^\Pi_z)$
with $G^\Pi_z$ being an isomorphism of the  eigensubspaces.
 \item[(d)] If $f\in \cH$, then $f\in \ran (A_{\Pi,\Theta}-z)$ iff $(G^\Pi_{\overline z})^*f\in \ran (\Theta-M^\Pi_z)$.
If $\Theta-M^\Pi_z$ is injective, the resolvent formula \eqref{krein2} still holds for
such $f$.
\end{itemize}
\end{corollary}

\subsection{Singular perturbations}
In this section let us recall a special approach to the construction
of boundary triplets as presented in \cite{Pos04} and \cite{Pos08} or in \cite[Section 1.4.2]{BGP08}.
Let $A$ be a self-adjoint operator in a Hilbert space $\cH$, then
we denote by $\cH_A$ the Hilbert space given by the linear space $\dom A$
endowed with the scalar product
$\langle u,v\rangle_{A}=\langle u,v\rangle_\cH+\langle A u,Av\rangle_\cH\,$.
Let $\fh$ be an auxiliary Hilbert space and
$\tau:\cH_A\rightarrow\mathfrak{h}$
be a bounded linear map which is \emph{surjective} and whose kernel
$\ker\tau$ \emph{is dense in } $\cH$, then the restriction $S$
of $A$ to $\ker \tau$ is a closed densely defined symmetric
operator in $\cH$. 
To simplify the formulas
we assume additionally that
\[
0\in\rho(A),
\]
which always holds in the subsequent applications.
For $z\in\rho(A)$ consider the maps
\begin{equation}
     \label{t-gamma}
G_z:=\big( \tau (A-\overline z)^{-1}\big)^*\in \cB(\fh,\cH),\quad M_z:=\tau\big( G_z - G_0)\equiv z\tau A^{-1}G_z \in \cB(\fh).
\end{equation}
\begin{proposition}
The adjoint $S^*$ is given by
\[
\dom S^*:=\big\{u=u_0+G_0 f_u: u_0\in \dom A \text{ and } f_u\in \fh\big\},
\quad
S^*u=Au_0. 
\]
Furthermore, the triplet $(\fh,\Gamma_1,\Gamma_2)$ with $\Gamma_1u:=f_u$
and $\Gamma_2 u:=\tau u_0$ is a boundary triplet for $S$, and 
the associated $\gamma$-field $G_z$ and Weyl function $M_z$
are given by \eqref{t-gamma}.
\end{proposition}

\begin{example}
\label{sum}

Let $A_\pm$ be self-adjoint operators in Hilbert spaces $\cH_\pm$
with $0\in\rho(A_\pm)$,
and let $\fh_\pm$, $\tau_\pm$, $S_\pm$, $G^\pm_z$, $M^\pm_z$, $\Gamma^\pm_j$
be the spaces and maps defined as above and associated with $A_\pm$.
For $\nu\in\RR\setminus\{0\}$ consider the operator $A:=A_-\osum \nu A_+$
acting in the Hilbert space $\cH:=\cH_-\osum \cH_+$.
Set $\tau=\tau_{-}\osum \nu\tau_{+}$, then
the restriction $S$ of $A$ to $\ker\tau$ has again the structure
of a direct sum, $S=S_{-}\oplus \nu S_{+}$, with $\gamma$-field and Weyl function
given by
\[
G_z= G^-_z\oplus G^+_{\frac{z}{\nu}},
\qquad
M_z=M^-_z \oplus \nu M^+_{\frac{z}{\nu}}.
\]
Thus, by the preceding considerations, the adjoint $S^*$
acts on the domain
\begin{equation}
   \label{domadj1}
\dom S^*= \big\{
u=(u_{-},u_{+}): \quad u_{\pm}=u^\pm_0+G_0^\pm\phi_\pm, \quad
u^\pm_0\in\dom A_{\pm},\quad \phi_{\pm}\in\fh_\pm
\big\}
\end{equation}
by $S^*(u_-,u_+)=A(u^-_0,u^+_0)$,
and one can take $(\fh_-\oplus\fh_+,\Gamma_1,\Gamma_2)$
as a boundary triplet for $S$,
\begin{equation}
      \label{ebt1}
\Gamma_1 u= (\phi_-,\phi_+),
\quad
\Gamma_2 u=(\tau_- u_0^-,\nu\tau_+u^+_0).
\end{equation}

\end{example}

\section{Boundary triplets for indefinite Laplacians}\label{sec3}

We start with some constructions for the closed symmetric operator 
\begin{equation}
       \label{eq-ss}
S=(-\Delta^{\min}_{-})\oplus\mu\Delta^{\min}_{+},\qquad
\Delta^{\min}_{\pm}:L^{2}(\Omega_{\pm})\supset H^{2}_{0}(\Omega_{\pm})
\to L^{2}(\Omega_{\pm}),
\end{equation}
where 
\begin{align*}
H^{2}_{0}(\Omega_{-})&:=\big\{u_{-}\in H^{2}(\Omega_{-}):(\gamma_0^{-}u_-,\gamma_1^{-}u_-)=(0,0) \big\}\,,\\
H^{2}_{0}(\Omega_{+})&:=\big\{u_{+}\in H^{2}(\Omega_{+}):(\gamma_0^{+}u_+,\gamma_1^{+}u_+,\gamma_{0}^{\partial}u_{+},\gamma_{1}^{\partial}u_{+})=(0,0,0,0)\big\}\,.
\end{align*}
Here and later on, $H^m(\Omega_{\pm})$ denotes the usual Sobolev-Hilbert
space of the square-integrable functions on $\Omega_{\pm}$ with square integrable partial
(distributional) derivatives of any order $k\le m$, and the linear operators 
\begin{align*}
\gamma_{0}^{\pm}:\,&H^2(\Omega_{\pm})\to
H^{\frac 32}(\Sigma)\,,&
\gamma_{1}^{\pm}:\,&H^2(\Omega_\pm)\to
H^{\frac 12}(\Sigma)\,,\\
\gamma_{0}^{\partial}:\,&H^2(\Omega_{+})\to
H^{\frac 32}(\partial\Omega)\,,&
\gamma_{1}^{\partial}:\,&H^2(\Omega_+)\to
H^{\frac 12}(\partial\Omega)\,,
\end{align*}
are the usual trace maps first defined on $u_{\pm}\in C^{\infty}(\overline{ \Omega_{\pm}})$ by 
\begin{align*}
\gamma_{0}^{\pm}u_{\pm}(x)&:=u_{\pm}(x),&\quad
\gamma_{1}^{\pm}u_{\pm}(x)&:=
N_{\pm}(x)\!\cdot\!\nabla u_{\pm}(x), && \quad x\in\Sigma,\\
\gamma_{0}^{\partial}u_{+}(x)&:=u_{+}(x),&\quad
\gamma_{1}^{\partial}u_{+}(x)&:=
N_{\partial}(x)\!\cdot\!\nabla u_{+}(x),
&&
\quad
x\in\partial\Omega\,,
\end{align*}
with $N_{\partial}$ being the outer unit normal on $\partial\Omega$,
and then extended by continuity.
It is well known, see e.g. \cite[Chapter 1, Section 8.2]{LM}, that the maps
\begin{align*}
H^{2}(\Omega_{-})\ni u_-&\mapsto  (\gamma_{0}^{-}u_{-},\gamma_{1}^{-}u_{-})\in H^{\frac 32}(\Sigma)\oplus H^{\frac12}(\Sigma)\,,\\
H^{2}(\Omega_{+})\ni u_+&\mapsto (\gamma_{0}^{+}u_{+},\gamma_{1}^{+}u_{+},\gamma_{0}^{\partial}u_{+},\gamma_{1}^{\partial}u_{+})
\in H^{\frac32}(\Sigma)\oplus H^{\frac 12}(\Sigma)\oplus H^{\frac 32}(\partial\Omega)\oplus H^{\frac 12}(\partial\Omega)
\end{align*}
are bounded and surjective.

We remark that both $\Sigma$ and $\partial\Omega$ can be made smooth compact Riemannian manifolds.
For $\Xi=\Sigma$ or $\Xi=\partial\Omega$, the fractional order Sobolev-Hilbert spaces $H^s(\Xi)$
with $s\in\RR$, are defined in the standard way
as the completions of $C^\infty(\Xi)$ with respect to the scalar
products 
\[
\langle \phi_{1},\phi_{2}\rangle_{H^s(\Xi)}
:=\big\langle \phi_{1},(-\Delta_{\Xi}+1)^{s}\phi_{2}\big\rangle_{L^2(\Xi)},
\]
where $\Delta_{\Xi}$ is the (negative) Laplace-Beltrami operator in 
$L^2(\Xi)$, see e.g. \cite[Remark 7.6, Chapter 1, Section 7.3]{LM},
and then $(-\Delta_{\Xi}+1)^{\frac s2}$ extends to a unitary map
from $H^{r}(\Xi)$ to $H^{r-s}(\Xi)$. In what follows we denote for shortness
\[
\Lambda:=\sqrt{-\Delta_{\Sigma}+1}, \quad
\Lambda_{\partial}:=\sqrt{-\Delta_{\partial\Omega}+1}.
\]
By Green's formula, the linear operators $\gamma_{0}^{\pm}$, $\gamma_{1}^{\pm}$, $\gamma_{0}^{\partial}$ and 
$\gamma_{1}^{\partial}$ can be then
extended to continuous (with respect to the graph norm) maps
\begin{equation}
    \label{eq-g01}
\begin{aligned}
\gamma_{0}^{\pm}&:\dom\Delta^{\max}_{\pm}\to H^{-\frac 12}(\Sigma)\,,&
\gamma_{1}^{\pm}&:\dom \Delta^{\max}_{\pm}\to H^{-\frac 32}(\Sigma)\,,\\
\gamma_{0}^{\partial}&:\dom \Delta^{\max}_+\to H^{-\frac 12}(\partial\Omega)\,,&
\gamma_{1}^{\partial}&:\dom\Delta^{\max}_+\to H^{-\frac 32}(\partial\Omega)\,,
\end{aligned}
\end{equation}
where $\Delta^{\max}_{\pm}:=(\Delta^{\min}_{\pm})^{*}$ acts as the distributional Laplacian on the domain 
\[
\dom \Delta^{\max}_{\pm}:=\big\{u_{\pm}\in L^2(\Omega_{\pm})\,:\, \Delta u_{\pm}\in L^2(\Omega_{\pm})\big\}\,,
\] 
see \cite[Chapter 2, Section 6.5]{LM}.
Now consider the operator
\[
A=(-\Delta^{D}_-)\oplus\mu\Delta^{D}_+
\]
acting in $L^{2}(\Omega)\equiv L^2(\Omega_{-})\oplus L^2(\Omega_{+})$,
where $\Delta^{D}_\pm$ are  the Dirichlet Laplacians in $L^2(\Omega_\pm)$, i.e.
\begin{align*}
\dom \Delta^{D}_-&=\big\{u_{-}\in H^2(\Omega_{-})\,:\, \gamma^{-}_{0}u_{-}=0\big\}\,,\\
\dom \Delta^{D}_+&=\big\{u_{+}\in H^2(\Omega_{+})\,:\, (\gamma^{+}_{0}u_{+},\gamma^{\partial}_{0}u_{+})=(0,0)\big\}\,.
\end{align*}
As both $\Delta^D_\pm$ are self-adjoint with compact resolvents, the same applies to $A$. The maps
\begin{align*}
\tau_{-}&:\dom \Delta^{D}_- \to H^{\frac 12}(\Sigma)\,, &\tau_{-}u_{-}&:=\gamma_{1}^{-}u_{-}\,,\\
\tau_{+}&:\dom \Delta^{D}_+\to H^{\frac 12}(\Sigma)\osum H^{\frac 12}(\partial\Omega)\,, &\tau_{+}u_{+}&:=(\gamma_{1}^{+}u_{+},\gamma_{1}^{\partial}u_{+})\,,
\end{align*}
are linear, continuous, surjective, and their kernels are dense in $L^{2}(\Omega_\pm)$.
Moreover, $\Delta^{\min}_\pm$ is exactly the restriction of $\Delta^{D}_\pm$ to $\ker \tau_{\pm}$.
Therefore, we may use the construction of Example~\ref{sum} with $\nu=-\mu$ to obtain a description
of the self-adjoint extensions of $S$ from \eqref{eq-ss}.
To this end, an expression for the associated operators $G_{z}^{\pm}$
and $M^{\pm}_{z}$ is needed. These were already obtained in
\cite[Example 5.5]{Pos08}, and we recall the final result.
The Poisson operators
\[
P_{z}^{-}:H^{s}(\Sigma)\to \dom \Delta^{\max}_-\,,\qquad
P_{z}^{+}:H^{s}(\Sigma)\osum H^{s}(\partial\Omega)\to \dom\Delta^{\max}_+
\,,\qquad s\ge -\frac12\,,
\]
are defined through the solutions of the respective boundary value problems,
\begin{align*}
P_{z}^{-} \phi = f &\text{ iff }
\left\{
\begin{aligned}
-\Delta^{\max}_- f &=zf,\\
\gamma_{0}^{-} f&=\phi\,,
\end{aligned}
\right., \quad z\in \rho(-\Delta^D_-),\\
P_z^+(\varphi,\psi)=g & \text{ iff }
\left\{
\begin{aligned}
-\Delta^{\max}_+g&=zg,\\
\gamma_{0}^{+} g&=\varphi\\
\gamma_{0}^{\partial} g&=\psi.
\end{aligned}
\right.,
 \quad z\in \rho(-\Delta^D_+),
\end{align*}
and the associated (energy-dependent)
Dirichlet-to-Neumann operators are given by
\begin{align*}
D_{z}^{-}&:H^{s}(\Sigma)\to 
H^{s-1}(\Sigma)\,,\quad
D^{-}_{z}:=\gamma_{1}^{-} P^{-}_{z}\,,\quad s\ge -\frac12\,,\\
D_{z}^{+}&:H^{s}(\Sigma)\oplus H^{s}(\partial\Omega)\to 
H^{s-1}(\Sigma)\oplus H^{s-1}(\partial\Omega)
\,,\quad s\ge -\frac12\,,\\
&D^{+}_{z}(\varphi,\psi):=\Big(
\gamma_{1}^{+} P^{+}_{z}(\varphi,\psi),
\gamma_{1}^{\partial} P^{+}_{z}(\varphi,\psi)
\Big)\,,
\end{align*}
then 
\begin{align*}
G_{z}^{-}&=-P^{-}_{z}\Lambda\,, & M_{z}^{-}&=(D^{-}_{0}-D_{z}^{-})
\Lambda\,,\\
G_{z}^{+}&=-P^{+}_{z}(\Lambda\oplus\Lambda_{\partial})\,,&
M_{z}^{+}&=(D^{+}_{0}-D_{z}^{+})(\Lambda\oplus\Lambda_{\partial}).
\end{align*}

Thus, by Remark \ref{sum}, the adjoint $S^*$ acts as
$S^*u=(-\Delta^{\max}_-u_{-}, \mu \Delta^{\max}_+ u_{+})$
on the domain $\dom (\Delta^{\max}_-\oplus \Delta^{\max}_+)$,
and using \eqref{domadj1} and \eqref{ebt1} one obtains
the boundary triplet $(\fh,\Gamma_1,\Gamma_2)$ for $S$
with $\fh=H^{\frac 12}(\Sigma)\oplus H^{\frac 12}(\Sigma)\oplus H^{\frac 12}(\partial\Omega)$
and
\[
\Gamma_1 u=-
\begin{pmatrix}
\Lambda^{-1}\gamma^{-}_{0}u_{-}\\
\Lambda^{-1}\gamma^{+}_{0}u_{+}\\
\Lambda^{-1}_\partial\gamma^{\partial}_{0}u_{+}
\end{pmatrix},
\quad
\Gamma_2u=
\begin{pmatrix}
\gamma^{-}_{1}(u_{-}-P_0^{-}\gamma_{0}^{-}u_{-})\\
-\mu\gamma_{1}^{+}\big(u_{+} -P_0^{+}(\gamma_{0}^{+}u_{+},\gamma_{0}^{\partial}u_{+})\big)\\
-\mu\gamma_{1}^{\partial}\big(u_{+}
-P_0^{+}(\gamma_{0}^{+}u_{+},\gamma_{0}^{\partial}u_{+})\big)
\end{pmatrix}.
\]
The associated $\gamma$-field $G_z$ and $M_z$ are given by
\[
G_z\begin{pmatrix}
\varphi_-\\
\varphi_+\\
\varphi_\partial
\end{pmatrix}
=
-
\begin{pmatrix}
P^-_z\Lambda \varphi_-\\
P^+_{-\frac{z}{\mu}}(\Lambda \varphi_+,\Lambda_\partial\varphi_\partial)
\end{pmatrix},
\quad
M_z
\begin{pmatrix}
\varphi_-\\
\varphi_+\\
\varphi_\partial
\end{pmatrix}
=
\begin{pmatrix}
(D^-_0-D^-_z)\Lambda \varphi_-\\
-\mu\big(D^+_0-D^+_{-\frac{z}{\mu}}\big)(\Lambda \varphi_+,\Lambda\varphi_\partial)
\end{pmatrix}.
\]

Let us represent the operator $L_\mu$ given by \eqref{opl} and \eqref{doml}
in the form $A_{\Pi,\Theta}$. Remark first that, in view of the elliptic regularity,
see e.g. \cite[Proposition III.5.2]{Grubb}, we have
\begin{align*}
H^2(\Omega_-)&=\big\{u_-\in L^2(\Omega_-): \Delta^{\max}_-u_-\in L^2(\Omega_-), \ \gamma^+_0u_-\in H^{\frac 32}(\Sigma)\big\}\\
H^2(\Omega_+)&=\big\{u_+\in L^2(\Omega_+): \Delta^{\max}_+u_+\in L^2(\Omega_+), \quad (\gamma^+_0 u_+,\gamma^\partial_0 u_+)\in H^{\frac 32}(\Sigma)\osum H^{\frac 32}(\partial\Omega)\big\}.
\end{align*}
Therefore, $L_\mu$ is exactly the restriction of $S^*$ to the functions $u=(u_-,u_+)$ with
\begin{equation}
      \label{eq-lll}
\quad \gamma^-_0u_-=\gamma^+_0u_+=:\gamma_0 u,
\quad \gamma^\partial_0u_+=0,
\quad
\gamma_0u\in H^{\frac 32}(\Sigma),
\quad
\gamma^-_1 u_-=\mu \gamma^+_1 u_+.
\end{equation}
The first two conditions can be rewritten as $\Gamma_1 u\in \ran \Pi$, where
$\Pi$ is the orthogonal projector in $\fh$ given by
\[
\Pi(\varphi_-,\varphi_+,\varphi_\partial)=\dfrac{1}{2}(\varphi_-+\varphi_+,\varphi_-+\varphi_+,0)\,.
\]
 For the subsequent computations it is useful to introduce the unitary operator
\[
U:\ran \Pi\to H^{\frac{1}{2}}(\Sigma), \quad
U(\varphi,\varphi,0)=\sqrt{2}\,\varphi,
\]
then
\[
U\Pi\Gamma_2u=\dfrac{1}{\sqrt 2}\Big[\gamma^{-}_{1}(u_{-}-P_0^{-}\gamma_{0}^{-}u_{-})
-\mu\gamma_{1}^{+}\big(u_{+} -P_0^{+}(\gamma_{0}^{+}u_{+},\gamma_{0}^{\partial}u_{+})\big)\Big],
\]
and the third and the fourth conditions in \eqref{eq-lll} can be rewritten respectively as
\[
\Gamma_1 u\in U^* \dom \Theta_\mu, \quad
U\Pi \Gamma_2 u= \Theta_\mu U\Gamma_1 u,
\]
where   $\Theta_\mu$ is the symmetric operator in $H^{\frac{1}{2}}(\Sigma)$ given by
\[
\Theta_\mu:=\dfrac{1}{2}(D^-_0-\mu \Tilde D^+_0)\Lambda,
\quad
\dom \Theta_\mu=H^{\frac52}(\Sigma),
\]
with
\[
\Tilde D^+_z:=\gamma_1^+ \Tilde P^+_z,
\quad
\Tilde P^+_z:=P^+_z(\cdot, 0).
\]
Therefore, one has the representation $L_{\mu}=A_{\Pi,U^*\Theta_\mu U}$,
and, due to the unitarity of $U$ and to the discussion of Section~\ref{sec2},
the operator $L_\mu$ is self-adjoint/essentially self-adjoint in $L^2(\Omega)$
if and only if $\Theta_\mu$ has the respective
property as an operator in $H^{\frac 12}(\Sigma)$.
Remark that for the associated maps $G^\Pi_z:=G_z\Pi^{*} $
and $M^\Pi_z:=\Pi M_z \Pi^{*} $ (see Subsection~\ref{secbt}) one has
\begin{equation}
           \label{agm}
G^\Pi_z U^*\varphi
=-\dfrac{1}{\sqrt 2}\begin{pmatrix}
P^-_z\Lambda \varphi\\[\medskipamount]
\Tilde P^+_{-\frac{z}{\mu}}\Lambda \varphi
\end{pmatrix},
\quad
U M^\Pi_zU^* =\dfrac{1}{2}\Big(
(D^-_0-D^-_z)-\mu(\Tilde D^+_0-\Tilde D^+_{-\frac{z}{\mu}})
\Big) \Lambda.
\end{equation}

\section{Proofs of main results}\label{sec4}

With the above preparations, the proofs will be based on an application 
of the theory of pseudodifferential operators, see e.g. \cite{T} and \cite{Taylor}.
At first we recall some  known results adapted to our setting.

If $\Psi\in \cB\big(H^{s}(\Sigma), H^{s-k}(\Sigma)\big)$ is a symmetric pseudodifferential operator of order $k$,
we set $k_0:=\max(k,0)$ and denote by $\Psi^{\min}$ and $\Psi^{0}$ the symmetric operators in $L^{2}(\Sigma)$ given
by the restriction of $\Psi$ to $\dom \Psi^{\min}=C^\infty(\Sigma)$ and $\dom \Psi^{0}=H^{k_0}(\Sigma)$ respectively,
then $\Psi^{0}\subseteq \overline {\Psi^{\min}}\subseteq \overline{\Psi^{0}}$.
Furthermore, if  $\Psi$ is elliptic, then  $\Psi^{0}$ is closed and, hence, $\overline {\Psi^{\min}}=\Psi^{0}$.
Since $\dom(\Psi^{\min})^{*}=\{f\in L^{2}(\Sigma):\Psi f\in L^{2}(\Sigma)\}$, for elliptic $\Psi$ one has $\dom(\Psi^{\min})^{*}\subseteq H^{k_0}(\Sigma)=\dom \overline {\Psi^{\min}}$, and so $\Psi^{\min}$ is essentially self-adjoint  and $\Psi^{0}$ is self-adjoint. It is important to
recall that for $k=1$ one does not need the ellipticity:
\begin{lemma}\label{ess}
If $\Psi$ is a symmetric first order pseudodifferential operator, then $\Psi^{\min}$, and then also $\Psi^0$, is essentially self-adjoint in $L^2(\Sigma)$.
\end{lemma}
\begin{proof}
By \cite[Proposition 7.4]{T}, for any $f\in L^{2}(\Sigma)$ with $\Psi f\in L^{2}(\Sigma)$
there exist $(f_{j})\subset C^{\infty}(\Sigma)$ such that $f_{j}\to f$ and $\Psi^{\min}f_{j}\to \Psi f$ in $L^{2}(\Sigma)$,
which literally means that $\dom (\Psi^{\min})^{*}\subseteq \dom \overline {\Psi^{\min}}$.
\end{proof}

In what follows, instead of studying $\Theta_\mu$ in $H^{\frac 12}(\Sigma)$ we prefer
to deal with the unitarily equivalent operator $\Phi_\mu:=\Lambda^{\frac 12}\Theta_\mu \Lambda^{-\frac 12}$
acting in $L^2(\Sigma)$. Set
\[
\Psi_{\mu}:=\frac12\,\Lambda^{{\frac12}}(D^{-}_{0}-\mu\Tilde D^{+}_{0})\Lambda^{{\frac12}},
\]
then $\Phi_\mu$ is the restriction of $\Psi_\mu$ to $\dom \Phi_\mu=H^2(\Sigma)$.
Furthermore, denote by $\Psi^{\min}_{\mu}$ and $\Psi_{\mu}^{0}$ the symmetric operators in $L^{2}(\Sigma)$  given respectively by the restrictions of $\Psi_{\mu}$
to $\dom\Psi_{\mu}^{\min}=C^{\infty}(\Sigma)$ and to $\dom\Psi_{\mu}^{0}=H^{k_0}(\Sigma)$,
where $k$ is the order of $\Psi_{\mu}$ and $k_0=\max(k,0)$. Remark that we always have $k\le 2$, hence,
$\Psi_\mu^{\min}\subseteq \Phi_\mu\subseteq \Psi_\mu^0$.

\begin{proof}[\bf Proof of Theorem~\ref{thm1}]
Assume first that $\mu\ne 1$. Let us show that the operator $\Theta_\mu$ is self-adjoint in $H^{\frac 12}(\Sigma)$, then this will imply
the self-adjointness of  $L_{\mu}$ in $L^2(\Omega)$. It is a classical result  that $D^\pm_0$ are first order pseudodifferential
operators with the principal symbol $|\xi|$, see e.g. \cite[Chapter 7, Section 11]{Taylor},
and, in view of the definition, the same applies then to $\Tilde D^+_0$.
Then $\Psi_\mu$ is a pseudodifferential operator with principal symbol $\frac{1-\mu}{2}|\xi|^2$.
As such a principal symbol is non-vanishing, $\Psi_\mu$ is a second order 
elliptic pseudodifferential operator and, by the results recalled at the beginning of the section,
$\Phi_\mu\equiv\Psi^{0}_{\mu}$ is self-adjoint on the domain $H^{2}(\Sigma)$.
Hence, since $\Lambda^{{\frac12}}:H^{\frac 12}(\Sigma)\to L^2(\Sigma)$ is unitary,  the operator
$\Theta_{\mu}=\Lambda^{-\frac12}\Phi_\mu\Lambda^{\frac12}$ is self-adjoint on the initial domain
$H^{{\frac52}}(\Sigma)$, which implies the self-adjointness of $L_\mu$ on the initial domain
$\cD^2_\mu(\Omega\setminus\Sigma)$.
Due to \eqref{cdd1} we have $\dom L_\mu\subseteq H^1_0(\Omega)$,
and the compact embedding of $H^1_0(\Omega)$ into $L^2(\Omega)$ proves that the resolvent of $L_\mu$
is a compact operator.

Let $\mu=1$, then $\Psi_1$ is a first order pseudodifferential operator, and $\Psi^{\min}_1$
is essentially self-adjoint due to Lemma~\ref{ess}. Then $\Phi_1$ is also
essentially self-adjoint being a symmetric extension of $\Psi^{\min}_1$.
The unitarity of $\Lambda^{{\frac12}}:H^{\frac12}(\Sigma)\to L^2(\Sigma)$
implies the essential self-adjointness of $\Theta_1$ in $H^{\frac 12}(\Sigma)$
and, in turn, that of $L_1$ in $L^2(\Omega)$.
\end{proof}

Recall that in what follows we denote by $\cL_1$ the unique self-adjoint extension of $L_1$.
In view of the discussion of Section~\ref{sec3} one has   $\cL_1=A_{\Pi,U^*\Theta U}$ 
with $\Theta=\overline{\Theta_1}$ being the closure (and the unique self-adjoint extension) of $\Theta_1$
in $H^{\frac 12}(\Sigma)$.

\begin{proof}[\bf Proof of Theorem~\ref{thm2}]
Assume that $n=2$ and $\mu=1$, then $\Psi_1=\frac{1}{2}\Lambda^{{\frac12}}(D^-_0-\Tilde D^+_0)\Lambda^{{\frac12}}$.
It is well known that the \emph{full} symbol of the classical Dirichlet-to-Neumann map  (at $z=0$) on a smooth bounded
two-dimensional domain with respect to the arclength is equal to $|\xi|$, see \cite[Proposition 1]{edw} for a direct proof
or \cite[Section 1]{uhl} for an iterative computation.
It follows that $\Psi_1$ is a symmetric pseudodifferential operator of order $(-\infty)$,
hence $\Psi^{0}_{1}:L^{2}(\Sigma)\to L^{2}(\Sigma)$ is bounded, self-adjoint and compact. 
As $\Phi_1$ is densely defined, it follows that $\overline{\Phi_1}=\Psi^0_1$.
Since $\Lambda^{{\frac12}}:H^{\frac 12}(\Sigma)\to L^2(\Sigma)$ is unitary, 
the closure of $\Theta_1$ is given by $\Theta=\Lambda^{-{\frac12}}\Phi_1\Lambda^{{\frac12}}: H^{{\frac12}}(\Sigma)\to H^{{\frac12}}(\Sigma)$,
and it is a compact self-adjoint operator in $H^{\frac 12}(\Omega)$.
As $\dom \Theta=H^{\frac 12}(\Sigma)$, the boundary condition $\Gamma_1 u\in U^*\dom \Theta$ takes the form
$\gamma^-_0 u_-=\gamma^+_0 u_+\in H^{-\frac 12}(\Sigma)$, $\gamma^\partial_0u_+=0$,
and, in view of \eqref{eq-g01}, the domain of $\cL_1=A_{\Pi,U^*\Theta U}$ is given by \eqref{dom2}.

Let us study the spectral properties of $\cL_{1}$ using Corollary~\ref{corol5}.
As $U^*\Theta U - M^\Pi_0\equiv U^*\Theta U$ is compact, one has $0\in\sigma_\ess(U^*\Theta U - M^\Pi_0)$ implying $0\in\sigma_\ess(\cL_1)$.
To prove the reverse inclusion $\sigma_\ess(\cL_1)\subseteq\{0\}$ we note first
that the operators $U^*\Theta U-M^\Pi_z$ are unitarily equivalent to $\Theta-UM^\Pi_zU^*$
and, hence, have the same spectra. Furthermore, 
the principal symbol of $D^\pm_0-D^\pm_\lambda$ is $\frac{\lambda}{2 |\xi|}$ for any $\lambda\in\rho(-\Delta^D_\pm)$, 
see \cite[Lemma 1.1]{LV}. As the principal symbol of $\Lambda$ is $|\xi|$, it follows that, for any $z\in \rho(A)$, the operators
$(D^-_0-D^-_z)\Lambda$ and $(\Tilde D^+_0-\Tilde D^+_{-\frac{z}{\mu}})\Lambda$ are bounded in $H^{\frac12}(\Sigma)$
being pseudodifferential operators of order zero, and their principal symbols are $\frac{z}{2}$ and $\big(-\frac{z}{2\mu}\big)$ respectively.
By Eq.~\eqref{agm}  it follows that the principal symbol of $\Theta-U M^\Pi_zU^*$ is simply $\frac{z}{2}$,
and one can represent  $\Theta-U M^\Pi_zU^*=\frac{z}{2} +K_z$,  where
$K_z$ are compact operators depending holomorphically on $z\in \rho(A)$.
As the operator $A$ has compact resolvent, it follows by \eqref{gam1} that
the only possible singularities of $z\mapsto K_z$ at the points of $\sigma(A)$
are simple poles with finite-dimensional residues.
Therefore, the operator function  $z\mapsto U^*\Theta U-M^\Pi_z:=U^*(\Theta-U M^\Pi_zU^*)U$ 
satisfies the assumptions of the meromorphic Fredholm alternative on $\CC_0:=\CC\setminus\{0\}$,
see \cite[Theorem XIII.13]{RS4}, and either
(a)  $0\in\sigma(U^*\Theta U-M^\Pi_z)$  for all $z\in \CC_0\setminus\rho(A)$,
or (b) there exists a subset $B\subset \CC_0$, without accumulation points
in $\CC_0$, such that the inverse  $(U^*\Theta U-M^\Pi_z)^{-1}$  exists and is bounded for
$z\in \CC_0\setminus \big(B \mathop{\cup}\sigma(A)\big)$ and extends to a meromorphic function
in $\CC_0\setminus B$ such that the coefficients in the Laurent series
at the points of $B$ are finite-dimensional operators.
The case (a) can be excluded: By Corollary~\ref{corol5}
this would imply the presence of a non-empty non-real spectrum
for $\cL_{1}$, which is not possible due to the self-adjointness.
Therefore, we are in the case (b), and the resolvent formula \eqref{krein2}
 for $\cL_1\equiv A_{\Pi,U^*\Theta U}$ 
implies that the set $\CC_0\cap\sigma(\cL_{1})\cap \rho(A)\subseteq B$
has no accumulation points in $\CC_0$,
and each point of this set is a discrete eigenvalue of $\cL_{1}$.
Furthermore, by \eqref{gam1} the maps $z\mapsto G^\Pi_z$
can have at most simple poles with finite-dimensional residues
at the points of $\sigma(A)$, and it is seen again from the resolvent formula~\eqref{krein2}
that the only possible singularities of $z\mapsto (\cL_{1}-z)^{-1}$ at the points of $\sigma(A)$
are poles with finite-dimensional residues. It follows that each point of $\sigma(A)$
is either not in the spectrum of $\cL_{1}$ or is its discrete eigenvalue. Therefore, $\cL_1$ has no essential spectrum in $\CC\setminus\{0\}$,
and the only possible accumulation points for the discrete  eigenvalues are $0$ and $\infty$.
\end{proof}

\begin{proof}[\bf Proof of Theorem~\ref{thm3}]
Assume $n\ge 3$ and $\mu=1$, then again $\Psi_1=\frac{1}{2}\Lambda^{{\frac12}}(D^-_0-\Tilde D^+_0)\Lambda^{{\frac12}}$.
By \cite[Chapter 12, Proposition C.1]{Taylor}, there holds
\[
D^-_0=(-\Delta_\Sigma)^{\frac12}+B^- +C^-,
\quad
\Tilde D^+_0=(-\Delta_\Sigma)^{\frac12}+B^+ +C^+,
\]
where $C^\pm$ are pseudodifferential operators of order $(-1)$
and $B^\pm $ are pseudodifferential operator of order $0$
whose principal symbols are $\pm b_0(x,\xi)$, with 
\[
b_0(x,\xi)=\dfrac{1}{2}\left( \tr W_x - \dfrac{\langle \xi,W_{x}^*\xi\rangle_{T^*_x \Sigma}}{\langle \xi,\xi\rangle_{T^*_x \Sigma}}\right)
\]
and $W_x:=\mathrm{d}N_-(x):T_x\Sigma\to T_x\Sigma$ being the Weingarten map
and $W^*_x$ its adjoint. Therefore, $\Psi_1$ is a pseudodifferential
operator of order $1$ whose principal symbol is $\frac{1}{2}b_0(x,\xi)|\xi|\,$. 
As already seen, $\Psi^{\min}_{1}$ is then essentially self-adjoint by Lemma \ref{ess},
and, as before, $L_{1}$ is essentially self-adjoint and its self-adjoint closure is  $A_{\Pi,U^*\Theta U}$, 
where $\Theta:=\overline\Theta_{1}$. As $\Theta_1$ is a first order operator, one has
$H^{\frac 32}(\Sigma)\subseteq\dom \Theta$. In particular, the boundary condition
$\Gamma_1 u\in U^* H^{\frac 32}(\Sigma)$ 
entails $\gamma^-_0 u_-=\gamma^+_0u_+\in H^{{\frac12}}(\Sigma)$
and $\gamma^\partial_0u_+=0$. Due to the elliptic regularity,
see e.g. \cite[Chapter 2, Section 7.3]{LM}, this can be rewritten as $u\in H^1_0(\Omega)$
and gives the inclusion \eqref{dom3}. 

(a) Recall that the principal curvatures $k_1(x),\dots,k_{n-1}(x)$ of $\Sigma$ at a point $x$
are the eigenvalues of $W_x$, hence,
\[
\dfrac{1}{2}\Big(
k_1(x)+\dots+k_{n-1}(x)-\max_j k_j(x)
\Big)
\le b_0(x,\xi)
\le
\dfrac{1}{2}\Big(k_1(x)+\dots+k_{n-1}(x)-\min_j k_j(x)\Big).
\]
Let $\Sigma'$ be a maximal connected component of $\Sigma$. If all $k_j$ are either all strictly positive
or all strictly negative on $\Sigma'$, one can estimate $a_1\le \big|b_0(x,\xi)\big|\le a_2$ for all $x\in\Sigma'$
with some $a_1>0$ and $a_2>0$. Therefore, in this case $\Psi_1$ is a first order elliptic pseudodifferential operator and so,
by the results recalled at the beginning of this section, $\Psi^{0}_{1}$ is self-adjoint. This implies that $\dom \Theta\equiv \dom \overline{\Theta_1}=H^{\frac 32}(\Sigma)$.
As before, the boundary condition  $\Gamma_1u\in U^*\dom \Theta$  for $u\in\dom \cL_1$
entails $u\in H^1_0(\Omega)$, and one arrives at the equality \eqref{dom4}. The inclusion $\dom \cL_1\subseteq H^1_0(\Omega)$
and the compact embedding of $H^1_0(\Omega)$ into $L^2(\Omega)$ imply that $\cL_1$ has compact resolvent.

(b)  As $M^\Pi_0=0$, by Corollary~\ref{corol5}(b) and by the unitarity of $U$, to get $0\in\sigma_{\ess}(\cL_{1})$ it suffices to show that $0\in\sigma_\ess(\Theta)$.
As $\Lambda^{\frac 12}:H^{\frac 12}(\Sigma)\to L^2(\Sigma)$
is a unitary operator, it is sufficient to show
$0\in\sigma_\ess(\overline{\Phi_1})$ for the unitarily equivalent
operator $\overline{\Phi_1}\equiv \Lambda^{\frac 12}\Theta \Lambda^{-\frac 12}$ in $L^2(\Sigma)$,
which will be done by constructing a singular Weyl sequence, i.e. a sequence of non-zero functions
$(u_j)\subset \dom \overline{\Phi_1}$ weakly converging to $0$ in $L^2(\Sigma)$
and such that the ratio $\|\Psi_1 u_j\|_{L^2(\Sigma)}/\|u_j\|_{L^2(\Sigma)}$ tends to $0$.
While the domain of $\overline{\Phi_1}$ is not known explicitly, we know already that
it contains $H^1(\Sigma)$.

Without loss of generality we assume that $\Sigma_\varepsilon:=\big\{(x',0): x'\in B_\varepsilon\big\}\subset \Sigma$,
where $B_\varepsilon$ is the ball in $\RR^{n-1}$ centered at the origin and of radius $\varepsilon>0$.
The iterative procedure of \cite[Section~1]{uhl} shows that the full symbols of $D^0_-$ and $\Tilde D^+_0$ on $\Sigma_\varepsilon$ in the local coordinates $x'$
are equal to $|\xi|$, and it follows that the full symbol of $\Psi_1$ vanishes on $\Sigma_\varepsilon$. Hence, there exists a smoothing
operator $K$ and $\delta\in(0,\varepsilon)$ such that
$\Psi_1 \Tilde u=K \Tilde u$ for all $u\in C^\infty_c(B_\delta)$,
where $\Tilde u$ is the extension of $u$ by zero to the whole of $\Sigma$.
Take an orthonormal sequence $(u_j)\subset L^2(B_\delta)$ with $u_j\in C^\infty_c(B_\delta)$,
then the sequence $(\Tilde u_j)\subset H^1(\Sigma)$ is orthonormal
and weakly converging to $0$ in $L^2(\Sigma)$. Due to the compactness of $K$ in $L^2(\Sigma)$ there exists a subsequence $(\Psi_1 \Tilde u_{j_k})$
strongly converging to zero in $L^2(\Sigma)$. Therefore, the sequence $v_k:=\Tilde u_{j_k}$
is a sought singular Weyl sequence for $\overline{\Phi_1}$, which gives the result.

Suppose now $\dom \cL_{1}=\cD^{s}_{1}(\Omega\backslash\Sigma)\subseteq H^{s}(\Omega_{-})\oplus H^{s}(\Omega_{+})$ for some $s>0$.
As the set on the right-hand side is compactly embedded in $L^2(\Omega)$, see e.g.~\cite[Theorem 14.3.1]{Agra},
this implies the compactness of the resolvent of $\cL_1$ and the equality $\sigma_\ess(\cL_1)=\emptyset$,
which contradicts the previously proved relation $0\in\sigma_\ess(\cL_1)$.
\end{proof}

\begin{remark}\label{rem10} 
After some simple cancellations, the resolvent formula of Corollary \ref{corol5} for $\cL_\mu$
takes the following form:
\[
(\cL_\mu-z)^{-1}\begin{pmatrix}u_-\\u_+\end{pmatrix}=
\begin{pmatrix}
(-\Delta^D_- -z)^{-1}u_-\\
(\mu \Delta^D_+-z)^{-1}u_+
\end{pmatrix}
-
\begin{pmatrix}
R^-_z(u_-,u_+)\\
R^+_z(u_-,u_+)
\end{pmatrix},
\]
with
\begin{align*}
R^-_z(u_-,u_+)&=P^-_z\Big( D^-_z-\mu\Tilde D^+_{-\frac{z}{\mu}}\Big)^{-1} \big(\gamma_1^-(-\Delta^D_- -z)^{-1}u_- - \mu \gamma_1^+(\mu\Delta^D_+ - z)^{-1}u_+\big),\\
R^+_z(u_-,u_+)&=P^+_{-\frac{z}{\mu}}\begin{pmatrix}
\Big( D^-_z-\mu\Tilde D^+_{-\frac{z}{\mu}}\Big)^{-1} \big(\gamma_1^-(-\Delta^D_- -z)^{-1}u_- - \mu \gamma_1^+(\mu\Delta^D_+ - z)^{-1}u_+\big)
\\
0
\end{pmatrix}.
\end{align*}

\end{remark}

\section{Proof of Proposition~\ref{thm4}}\label{sec5}

We continue using the conventions and notation introduced just before Theorem~\ref{thm4}.
In addition to \eqref{omb} we have
\[
\Omega_+=B_{r_i} \mathop{\cup} B_{r_e,R},
\quad
\Sigma=S_{r_i}\mathop{\cup}  S_{r_e},
\quad
\partial\Omega=S_R,
\]
and for the subsequent computations we use the identification
$L^2(S_\rho)\simeq L^2\big((0,2\pi), \rho \,d\theta\big)$, 
then $L^2(\Sigma)\simeq L^2\big((0,2\pi),r_i d\theta\big)\osum L^2\big((0,2\pi),r_e d\theta\big)$,
and similar identifications hold for the Sobolev spaces.

In view of Corollary~\ref{corol5} and of the expressions \eqref{agm},
the injectivity of $\cL_1$ is equivalent to the injectivity of the map
\[
\cD:=D^-_0-\Tilde D^+_0:H^{-\frac{1}{2}}(\Sigma)\to H^{\frac 12}(\Sigma),
\]
and then the condition $g=(0,g_+)\in \ran\cL_1$ is equivalent to
$(G^\Pi_0)^* g\equiv -\gamma^+_1 (-\Delta^D_+)^{-1}g_+\in \ran \cD$,
or, as $(\Delta^D_+)^{-1}:L^2(\Omega_+)\to H^2(\Omega)$
and $\gamma^+_1:H^2(\Omega)\to H^{\frac 12}(\Sigma)$, to 
\begin{equation}
   \label{eqd}
\cD^{-1}\gamma^+_1 (-\Delta^D_+)^{-1}g_+\in H^{-\frac 12}(\Sigma).
\end{equation}
The condition will be checked using an explicit computation
of the Dirichlet-to-Neumann maps $D^\pm_0$ and of the inverse of $\Delta^D_+$.
 
For a function $f$ defined in $\Omega_\pm$,  define its Fourier coefficients with respect to the polar angle
by
\[
f_m(r):=\dfrac{1}{2\pi}\int_0^{2\pi} f(r\cos\theta,r\sin\theta) e^{-\rmi m\theta} d\theta,
\quad
m\in\ZZ,
\]
then $f$ is reconstructed by $f(r\cos\theta,r\sin\theta)=\sum_{m\in\ZZ} f_m(r) e^{\rmi m\theta}$.
Furthermore, the separation of variables shows that
a function $f$ is harmonic iff $f_m$ satisfy the Euler equations
\[
f''_m(r)+r^{-1} f'_m(r)-m^2r^{-2}f_m(r)=0,
\]
whose linearly independent solutions are $1$ and $\ln r$ for $m=0$
and $r^{\pm m}$ for $m\ne 0$. This shows that for
\[
(\phi_i, \phi_e)\in H^s(\Sigma),
\quad
\phi_{i/e}(\theta)=\sum_{m\in\ZZ} \phi_{i/e,m} e^{\rmi m \theta},
\quad
\phi_{i/e,m}:=\dfrac{1}{2\pi}\int_0^{2\pi} \phi_{i/e}(\theta)e^{-\rmi m\theta}d\theta,
\]
one has the following expressions for the Poisson operators:
\begin{multline}
  \label{p1}
P^-_0 \begin{pmatrix}\phi_i\\ \phi_e\end{pmatrix}(r\cos\theta,r\sin\theta)=
\dfrac{\ln\dfrac{r_e}{r}}{\ln\dfrac{r_e}{r_i}}\,\phi_{i,0}
+
\dfrac{\ln\dfrac{r}{r_i}}{\ln\dfrac{r_e}{r_i}}\,\phi_{e,0}\\
+
\sum_{m\in \ZZ\setminus\{0\}}
\dfrac{
\bigg[
\Big(\dfrac{r_e}{r}\Big)^{|m|}
-
\Big(\dfrac{r}{r_e}\Big)^{|m|}
\bigg]\phi_{i,m}
+
\bigg[
\Big(\dfrac{r}{r_i}\Big)^{|m|}
-
\Big(\dfrac{r_i}{r}\Big)^{|m|}
\bigg]\phi_{e,m}
}{
\Big(\dfrac{r_e}{r_i}\Big)^{|m|}
-
\Big(\dfrac{r_i}{r_e}\Big)^{|m|}
}\, e^{\rmi m\theta},\\
(r,\theta)\in (r_i,r_e)\times(0,2\pi),
\end{multline}
and
\begin{multline}
       \label{p2}
\Tilde P^+_0 \begin{pmatrix}\phi_i\\ \phi_e\end{pmatrix}(r\cos\theta,r\sin\theta)\\
=\begin{cases}
\sum_{m\in \ZZ} \Big(\dfrac{r}{r_i}\Big)^{|m|} \phi_{i,m} e^{\rmi m\theta}, & (r,\theta)\in (0,r_i)\times(0,2\pi),\\
\dfrac{\ln\dfrac{R}{r}}{\ln\dfrac{R}{r_e}}\,\phi_{e,0}
+
\sum_{m\in \ZZ\setminus\{0\}}
\dfrac{\Big(\dfrac{R}{r}\Big)^{|m|}
-
\Big(\dfrac{r}{R}\Big)^{|m|}}{
\Big(\dfrac{R}{r_e}\Big)^{|m|}
-
\Big(\dfrac{r_e}{R}\Big)^{|m|}
}\,\phi_{e,m} e^{\rmi m\theta},
& (r,\theta)\in (r_e,R)\times(0,2\pi).
\end{cases}
\end{multline}
It follows that 
\begin{gather*}
D^-_0 \begin{pmatrix}\phi_i\\\phi_e\end{pmatrix}=\sum_{m\in\ZZ} B_m \begin{pmatrix}\phi_{i,m}\\\phi_{e,m}\end{pmatrix} e^{\rmi m \theta},
\quad
\Tilde D^+_0  \begin{pmatrix}\phi_i\\\phi_e\end{pmatrix}=\sum_{m\in\ZZ}  C_m \begin{pmatrix}\phi_{i,m}\\\phi_{e,m}\end{pmatrix}
e^{\rmi m \theta},\\
\cD \begin{pmatrix}
\phi_i\\
\phi_e
\end{pmatrix}
=\sum_{m\in\ZZ} D_m
\begin{pmatrix}
\phi_{i,m}\\
\phi_{e,m}
\end{pmatrix}
 e^{\rmi m \theta},
\quad
D_m:=B_m-C_m,
\end{gather*}
with
\begin{align*}
B_0&=\begin{pmatrix}
\dfrac{1}{r_i} \dfrac{1}{\ln\dfrac{r_e}{r_i}}
&
-\dfrac{1}{r_i} \dfrac{1}{\ln\dfrac{r_e}{r_i}}
\\[\medskipamount]
-\dfrac{1}{r_e} \dfrac{1}{\ln\dfrac{r_e}{r_i}}
&
\dfrac{1}{r_e} \dfrac{1}{\ln\dfrac{r_e}{r_i}}
\end{pmatrix},
\quad
C_0=\begin{pmatrix}
0 & 0\\[\medskipamount]
0 & \dfrac{1}{r_e}\, \dfrac{1}{\ln \dfrac{R}{r_e}}
\end{pmatrix},\\
B_m&= |m|\begin{pmatrix}
\dfrac{1}{r_i} \,\dfrac{\Big(\dfrac{r_e}{r_i}\Big)^{|m|}+\Big(\dfrac{r_i}{r_e}\Big)^{|m|}}{\Big(\dfrac{r_e}{r_i}\Big)^{|m|}-\Big(\dfrac{r_i}{r_e}\Big)^{|m|}}
&
-\dfrac{2}{r_i} \, \dfrac{1}{\Big(\dfrac{r_e}{r_i}\Big)^{|m|}-\Big(\dfrac{r_i}{r_e}\Big)^{|m|}}
\\[\medskipamount]
-\dfrac{2}{r_e} \, \dfrac{1}{\Big(\dfrac{r_e}{r_i}\Big)^{|m|}-\Big(\dfrac{r_i}{r_e}\Big)^{|m|}}
&
\dfrac{1}{r_e} \,\dfrac{\Big(\dfrac{r_e}{r_i}\Big)^{|m|}+\Big(\dfrac{r_i}{r_e}\Big)^{|m|}}{\Big(\dfrac{r_e}{r_i}\Big)^{|m|}-\Big(\dfrac{r_i}{r_e}\Big)^{|m|}}
\end{pmatrix}, \quad m\ne 0,\\
C_m&=|m|\begin{pmatrix}
\dfrac{1}{r_i} & 0\\
0 & 
\dfrac{1}{r_e}
\dfrac{\Big(\dfrac{R}{r_e}\Big)^{|m|}+\Big(\dfrac{r_e}{R}\Big)^{|m|}}{\Big(\dfrac{R}{r_e}\Big)^{|m|}
-\Big(\dfrac{r_e}{R}\Big)^{|m|}
}
\end{pmatrix},
\quad m\ne 0.
\end{align*}
Therefore,
\begin{align*}
D_0&=\begin{pmatrix}
\dfrac{1}{r_i} \dfrac{1}{\ln\dfrac{r_e}{r_i}}
&
-\dfrac{1}{r_i} \dfrac{1}{\ln\dfrac{r_e}{r_i}}
\\
-\dfrac{1}{r_e} \dfrac{1}{\ln\dfrac{r_e}{r_i}}
&
\dfrac{1}{r_e} \left(\dfrac{1}{\ln\dfrac{r_e}{r_i}}-\dfrac{1}{\ln \dfrac{R}{r_e}}\right)
\end{pmatrix},
\\
D_m&=2|m|\begin{pmatrix}
\dfrac{1}{r_i} \,\dfrac{1}{\Big(\dfrac{r_e}{r_i}\Big)^{2|m|}-1}
&
-\dfrac{1}{r_i} \,\dfrac{\Big(\dfrac{r_e}{r_i}\Big)^{|m|}}{\Big(\dfrac{r_e}{r_i}\Big)^{2|m|}-1}\\[\medskipamount]
-\dfrac{1}{r_e} \,\dfrac{\Big(\dfrac{r_e}{r_i}\Big)^{|m|}}{\Big(\dfrac{r_e}{r_i}\Big)^{2|m|}-1}
&
\dfrac{1}{r_e} \left(
\dfrac{1}{\Big(\dfrac{r_e}{r_i}\Big)^{2|m|}-1}-\dfrac{1}{\Big(\dfrac{R}{r_e}\Big)^{2|m|}-1}
\right)
\end{pmatrix},
\quad m\ne 0,
\end{align*}
hence, all $D_m$ are invertible, and then $\cD$ is injective with the inverse
\begin{equation}
   \label{dinv}
\cD^{-1} \begin{pmatrix}
\phi_i\\
\phi_e
\end{pmatrix}
=\sum_{m\in\ZZ} D_m^{-1}
\begin{pmatrix}
\phi_{i,m}\\
\phi_{e,m}
\end{pmatrix}
 e^{\rmi m \theta},
\end{equation}
which shows the injectivity of $\cL_1$. Furthermore, for $m\ne 0$ we have
\[
D_m^{-1}=-\dfrac{1}{2|m|}\begin{pmatrix}
r_i \bigg( 1-\Big( \dfrac{r_e^2}{r_i R}\Big)^{2|m|}\bigg) &
r_e  \Big(\dfrac{r_e}{r_i}\Big)^{|m|}\bigg( 1-\Big( \dfrac{r_e}{R}\Big)^{2|m|}\bigg)\\[\bigskipamount]
r_i  \Big(\dfrac{r_e}{r_i}\Big)^{|m|}\bigg( 1-\Big( \dfrac{r_e}{R}\Big)^{2|m|}\bigg)
&
r_e \bigg( 1-\Big( \dfrac{r_e}{R}\Big)^{2|m|}\bigg)
\end{pmatrix},
\]
and we conclude that a function $(\phi_i,\phi_e)\in H^{\frac 12}(\Sigma)$
belongs to $\ran\cD$ iff $\cD^{-1}(\phi_i,\phi_e)\in H^{-\frac 12}(\Sigma)$, i.e. iff
\begin{equation}
   \label{sum4}
\sum_{m\ne 0} \dfrac{1}{|m|} \left\| D_m^{-1} \begin{pmatrix}
\phi_{i,m}\\ \phi_{e,m}
\end{pmatrix}
\right\|^2_{\CC^2}<\infty.
\end{equation}
Therefore, the condition \eqref{eqd} is equivalent to \eqref{sum4}
for
\begin{equation}
    \label{phi1}
(\phi_i,\phi_e):=\gamma^+_1 f, \quad f:=(\Delta^D_+)^{-1} g_+.
\end{equation}
Remark first that $f$ vanishes in $B_{r_i}$, hence, $\phi_i=0$ and $f_m(r)=0$ for $r<r_i$ and $m\in\ZZ$.
To study the problem in $B_{r_e,R}$, let us pass to the Fourier coefficients,
then we arrive to the system of equations
\begin{equation}
    \label{eul2}
f''_m(r)+r^{-1}f'_m(r)-m^2 r^{-2}f_m(r)= h_m 1_{(a,b)}(r),
\quad
r_e<r<R,
\quad
f_m(r_e)=f_m(R)=0,
\end{equation}
and we have
\begin{equation}
   \label{phi3}
\begin{pmatrix}
\phi_i\\
\phi_e
\end{pmatrix}
=
-
\sum_{m\in\ZZ}
\begin{pmatrix}
0\\
 f_m'(r_e)
\end{pmatrix}
e^{\rmi m\theta}.
\end{equation}
One solves \eqref{eul2} using the variation of constants,
and for $m\ne 0$ the solutions are
\begin{gather*}
f_m(r)=\alpha_m r^m+\beta_m r^{-m} + \dfrac{h_m r^m}{2m}\int_{r_e}^r s^{1-m} 1_{(a,b)}(s)ds
-\dfrac{h_m r^{-m}}{2m}\int_{r_e}^r s^{1+m} 1_{(a,b)}(s) ds,\\
\alpha_m=-\dfrac{h_m}{2m r_e^m} \dfrac{1}{\Big( \dfrac{R}{r_e}\Big)^m-\Big( \dfrac{r_e}{R}\Big)^m}
\int_a^b \bigg(\Big( \dfrac{R}{s}\Big)^m-\Big( \dfrac{s}{R}\Big)^m\bigg) s\,ds,
\quad
\beta_m=-r_e^{2m} \alpha_m,
\end{gather*}
and
\[
f'_m(r_e)=\dfrac{m h_m}{r_e}\big( \alpha_m r_e^m-\beta_m r_e^{-m}\big)=
-\dfrac{1}{r_e}\,\dfrac{\displaystyle\int_a^b \bigg(\Big( \dfrac{R}{s}\Big)^{|m|}-\Big( \dfrac{s}{R}\Big)^{|m|}\bigg) s\,ds}{\Big( \dfrac{R}{r_e}\Big)^{|m|}-\Big( \dfrac{r_e}{R}\Big)^{|m|}}.
\]
Then for large $m$ one has
\[
f'_m(r_e)=-\dfrac{\big(a^2+o(1)\big)h_m}{r_e|m|}\, \Big( \dfrac{r_e}{a}\Big)^{|m|},
\quad
D_m^{-1}\begin{pmatrix} 0\\ - f'_m(r_e) \end{pmatrix}
=h_m\begin{pmatrix}
\dfrac{a^2+o(1)}{2m^2} \Big(\dfrac{r_e^2}{ r_i a}\Big)^{|m|}\\
\dfrac{a^2+o(1)}{2m^2} \Big(\dfrac{r_e}{a}\Big)^{|m|}
\end{pmatrix},
\]
and the condition \eqref{sum4} for the function \eqref{phi1} takes the form \eqref{eq-hm},
which finishes the proof.

One should remark that the condition \eqref{eq-hm} can still hold for $a<r_e^2/r_i$
if the Fourier coefficients $h_m$ of $h$ are very fast decaying for large $m$, i.e.
if $h$ extends to an analytic function in a suitable complex neighborhood of the unit circle.

\begin{remark} At last we note that, in view of the injectivity of $\cL_1$, the expression for its inverse given
in Remark~~\ref{rem10} can be extended naturally to a linear map
$\cL^{-1}:L^2(\Omega)\to \cD'(\Omega)$. As $\ran(\Delta^D_\pm)^{-1}= H^2(\Omega_\pm)\mathop{\cap}H^1_0(\Omega_\pm)$,
the finiteness of the norms $\|\cL^{-1}_1g\|_{H^s(V)}$, $V\subseteq\Omega$, $0\le s\le 1$,
is equivalent to the finiteness of $\|v\|_{H^s(V)}$ for $v:=\big(R^-_0g,R^+_0g\big)$.
The direct substitution of the values of \eqref{dinv} and \eqref{phi3} into \eqref{p1} and \eqref{p2}
shows that one always has $v\in H^1(B_{r_e,R})$,
while the condition $v\in L^2(B_{r_i,r_e})$ appears to be equivalent to \eqref{eq-hm}, so
it holds for any $h$ for $a\ge r_e^2/r_i$ as before, otherwise
a very strong regularity of $h$ is required.
\end{remark}

\appendix

\section{Proof of Proposition~\ref{prop1}(d)}\label{app1}

Let $f\in \ran(A_{1,\Theta}-z)$, then there is $g\in\dom A_{1,\Theta}$
with $f=(A_{1,\Theta}-z)g$. By \cite[Theorem 1.23(a)]{BGP08}, one can uniquely represent
\begin{equation}
  \label{ggz}
g=g_z+G_z h
\end{equation}
with $g_z\in\dom A$ and $h\in\fh$, and $f=(S^*-z)g=(A-z)g_z$. As $g_z\in \dom A=\ker \Gamma_1$,
we have $\Gamma_1 g=\Gamma_1 G_z h=h$ and $\Gamma_2g=\Gamma_2 g_z+\Gamma_2 G_z h$.
By \cite[Theorem 1.23(2d)]{BGP08} there holds
$\Gamma_2 g_z=\Gamma_2 (A-z)^{-1} f=G^*_{\overline z}f$, and by definition
we have $\Gamma_2 G_z h=M_z h$. Therefore, the boundary condition $\Gamma_2 g=\Theta\Gamma_1 g$
writes as $G^*_{\overline z} f=(\Theta-M_z)h$ implying
$G^*_{\overline z}f\in \ran (\Theta-M_z)$.
If $\Theta-M_z$ is injective, then $A_{1,\Theta}$ is also injective by
Proposition~\ref{prop1}(c), $h=(\Theta-M_z)^{-1}G^*_{\overline z}$, and
the substitution into \eqref{ggz} gives the relation
\[
(A_{1,\Theta}-z)^{-1}f=g=g_z+G_z h=(A-z)^{-1}f+G_z(\Theta-M_z)^{-1}G^*_{\overline z}f.
\]

Now let $f\in \cH$ such that $G^*_{\overline z}f\in\ran(\Theta-M_z)$. Take
$h\in\fh$ with $G^*_{\overline z}f=(\Theta-M_z)h$ and consider the function
$g=g_z+G_z h$ with $g_z=(A-z)^{-1}f\in \dom A$.
By \cite[Theorem 1.23(2d)]{BGP08} we have $g\in \dom S^*$.
As previously, $\Gamma_1 g=h$
and $\Gamma_2 g=G^*_{\overline z}f+M_z h=(\Theta-M_z)h+M_z h=\Theta h=\Theta \Gamma_1 g$,
i.e. $g\in \dom A_{1,\Theta}$, and we have $(A_{1,\Theta}-z)g=(S^*-z)g=(S^*-z)(A-z)^{-1}f=f$,
i.e. $f\in\ran(A_{1,\Theta}-z)$.

\end{document}